\documentclass[letterpaper, 10 pt, conference]{ieeeconf}  
\IEEEoverridecommandlockouts 
\usepackage{xcolor}

\usepackage[pdftex]{graphicx}
\graphicspath{{./Figs/}}
\DeclareGraphicsExtensions{.pdf}

\usepackage{cite, booktabs}
\usepackage{tikz}
\usetikzlibrary{calc, matrix, arrows}

\tikzstyle{block} = [draw, rectangle, thick,minimum height=2em,minimum width=5.5em]
\tikzstyle{branch} = [draw,circle,inner sep=0.5mm,fill=black]
\tikzstyle{none} = [draw=none]
\tikzstyle{connector} = [->,thick]
\tikzstyle{line} = [thick]

\usepackage[cmex10]{amsmath}
\usepackage{amsfonts,amssymb,amsthm}
\usepackage[caption=false,font=footnotesize]{subfig}
\usepackage[hidelinks]{hyperref}

\DeclareMathOperator{\cocl}{\overline{co}}
\DeclareMathOperator{\cl}{cl}
\DeclareMathOperator{\co}{co}
\DeclareMathOperator{\gph}{gph}

\newcommand{\calX}{\mathcal{X}}
\newcommand{\bbR}{\mathbb{R}}
\newcommand{\bbB}{\mathbb{B}}

\newtheorem{thm}{Theorem}
\newtheorem{definition}{Definition}
\newtheorem{lemma}{Lemma}
\newtheorem{prop}{Proposition}
\theoremstyle{remark}
\newtheorem{examplex}{Example}
\newtheorem{remarkx}{Remark}

\newcommand\oprocendsymbol{\hbox{\small $\blacksquare$}}
\newcommand\oprocend{\relax\ifmmode\else\unskip\hfill\fi\oprocendsymbol}
\newenvironment{example}{\examplex}{\oprocend\endexamplex}
\newenvironment{remark}{\remarkx}{\oprocend\endremarkx}

\title{\LARGE \bf
Time-varying Projected Dynamical Systems \\with Applications to Feedback Optimization of Power Systems
}

\author{Adrian Hauswirth, Irina Suboti\'c, Saverio Bolognani, Gabriela Hug, and Florian D\"orfler%
\thanks{The authors are with the Department of Information Technology and Electrical Engineering, ETH Zurich,
				8092 Zurich, Switzerland. Email:
				{\tt\small \{hadrian,subotici, bsaverio,ghug,dorfler\}@ethz.ch}. This work was supported by ETH Zurich funds and the SNF AP Energy Grant \#160573.}}

\begin{document}

\maketitle
\thispagestyle{empty}
\pagestyle{empty}

\begin{abstract}
This paper is concerned with the study of continuous-time, non-smooth dynamical systems which arise in the context of time-varying non-convex optimization problems, as for example the feedback-based optimization of power systems. We generalize the notion of projected dynamical systems to time-varying, possibly non-regular, domains and derive conditions for the existence of so-called Krasovskii solutions. The key insight is that for trajectories to exist, informally, the time-varying domain can only contract at a bounded rate whereas it may expand discontinuously. 
This condition is met, in particular, by feasible sets delimited via piecewise differentiable functions under appropriate constraint qualifications. To illustrate the necessity and usefulness of such a general framework, we consider a simple yet insightful power system example, and we discuss the implications of the proposed conditions for the design of feedback optimization schemes.
\end{abstract}


\section{Introduction}
The idea of ``closing the loop'' on a physical systems not just to control, but to optimize the state of a physical system with simple feedback controllers has recently emerged as a new frontier, combining ideas from  optimization and control theory with notable applications in the operation and optimization of power systems~\cite{nelson_integral_2017, bolognani_distributed_2015, bolognani_distributed_2013-1, gan_online_2016, dallanese_photovoltaic_2016, dallanese_optimal_2018,hauswirth_online_2017, tang_real-time_2017,bernstein_realtime_2018, hauswirth_projected_2016}. The underlying premise of such \emph{autonomous} (or \emph{feedback-based}) optimization schemes as illustrated in Fig.~\ref{fig:fb_loop} is that a nonlinear feedback controller induces closed-loop dynamics, usually in the form of simple gradient- or saddle-point flows~\cite{cherukuri_convergence_2015}, that steer a steady-state physical system to an optimal state.

A major challenge and a key difference to the traditional optimization context is the fact that the physical system enforces hard constraints on the evolution of the dynamical system. Physical conservation laws (expressed as equality constraints) constrain the system to a manifold, whereas saturation effects modelled by inequality constraints introduce non-smooth behavior. Furthermore, the feasible space enforced by the physical system is in general time-varying. These features expose fundamental questions regarding the mathematical modeling of such discontinuous systems and in particular the existence of viable solutions, i.e., solutions that adhere to the physical constraints.

\begin{figure}[tb]
\centering
    \resizebox{0.95\columnwidth}{!}{
    \begin{tikzpicture}
	\matrix[ampersand replacement=\&, row sep=0.6cm, column sep=1.5cm] {
		\node[block] (controller) {\begin{tabular}{c}
		   \footnotesize{Feedback} \\ \footnotesize{Optimizer}\end{tabular}}; \&  \node[block] (plant) {\begin{tabular}{c}
         Steady-state Plant \\ $h(x,u, t) = 0, \, g(x,u, t) \leq 0$ \end{tabular}};  \& \node[branch]  (b_1) {}; \\
         ; \& \node[none] (b_2) {};                \& ; \\		
	};
    \draw[connector] (controller.east)--(plant.west) node[midway, below] {$u$};    
    \draw[line] (plant.east)--(b_1.center);       
    \draw[connector] (b_1.center)--($(b_1.center)+(8mm,0)$) node[near start, above] {$x$};  
    \draw[line] (b_1.south)|-(b_2.center) node[midway, below] {}; 
    \draw[line] (b_2.center)-|($(controller.west)-(8mm,0mm)$);     
    \draw[connector] ($(controller.west)-(8mm,0mm)$)--(controller.west); 
  \end{tikzpicture}}
\caption{Feedback-based optimization} \label{fig:fb_loop}
\end{figure}
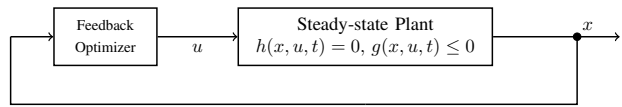

In this paper, we study the conditions required for the existence of physically realizable trajectories. As such, our findings are independent of any particular choice of feedback algorithm. In fact, we abstract the feedback controller by its induced vector field on the feasible domain. This leads us to consider \emph{projected dynamical systems}~\cite{nagurney_projected_1996, cornet_existence_1983, henry_existence_1973}, that are a natural choice to model physical processes involving saturation. 
For our analysis we draw inspiration from the study of switched hybrid systems~\cite{lygeros_dynamical_2003, liberzon_switching_2003}, non-smooth analysis~\cite{rockafellar_variational_1998}, and viability theory~\cite{aubin_viability_1991, aubin_viability_2011}. 

As a first contribution, we identify a Lipschitz-type requirement on the time-varying domain that is sufficient for the existence of solutions.
We then prove existence of so-called \emph{Krasovskii} solutions~\cite{hauswirth_projected_2018}. Despite its technical definition, this type of solution to differential equations is well-behaved, relatively easy to analyze and, most importantly, exists under very weak assumptions.

This level of generality is required since our work is motivated by the online optimization of power systems for which common modeling assumptions such as convexity or Clarke regularity fail.
Namely, the steady-state power grid is subject to the nonlinear, non-convex power flow equations. Furthermore, changes in power consumption and in the availability of renewable energy sources lead to a feasible set that changes continuously in time, but not in a differentiable way.
Finally, low-level, nonlinear controllers (e.g., frequency droop curves and automatic voltage regulation with reactive power limits) induce non-smooth steady-state behaviors. 
Discussion of detailed models for these domain-specific aspects is beyond the scope of this work. Nevertheless, at the end of the paper, we provide a highly stylized and deliberately simplified power systems example that captures the main challenges addressed with our approach.

The rest of the paper is structured as follows: In Section~\ref{sec:prelim} we introduce the notion of a \emph{temporal tangent cone} to a time-varying domain as generalization of a tangent cone to a stationary set. We establish that this temporal tangent cone is non-empty if the domain is \emph{forward Lipschitz continuous}. This enables us in Section~\ref{sec:pds} to define projected dynamical systems on time-varying domains and prove the existence of Krasovskii solutions under forward Lipschitz continuity of the domain. In Section~\ref{sec:ps_ex} we present a simplified example of the feasible domain of a power system that isolates the structural complications that are addressed by our results.

\section{Preliminaries \& Foundational Results} \label{sec:prelim}

We consider $\bbR^n$ endowed with the usual topology, canonical basis and Euclidean 2-norm $\| \cdot \|$, and we use $\mathbb{B} := \{ x \in \bbR^n \,|\, \| x \| \leq  1\}$ to denote the closed unit ball. Furthermore, we denote by $\cl A$ the closure of the set $A$ and by $\co A$ ($\cocl A$) its (closed) convex hull. 
Given a sequence $\{\delta_k\}$, the notation $\delta_k \rightarrow 0^+$ implies that $\delta_k$ converges to 0 and $\delta_k > 0$ for all $k$. 
A map $\Phi: \bbR^n \rightarrow \bbR^m$ is of class $C^k$ if it is $k$-times continuously differentiable. 
The Jacobian of $\Phi$ at $x$ is given by the $m\times n$-matrix $\nabla \Phi(x)$ of partial derivatives. The Jacobian at $x$ with respect to a variable $x'$ is denoted by $\nabla_{x'} \Phi(x)$.
A set-valued map from $U \subset \mathbb{R}^m$ to subsets of $\mathbb{R}^n$ is denoted by $F: U \rightrightarrows \mathbb{R}^n$. A set-valued map is \emph{non-empty}, \emph{closed}, \emph{compact} or \emph{convex} if $F(x)$ is non-empty, closed, compact or convex for every $x$ in its domain. 

\subsection{Generalization of the tangent cone}

Given a set $\calX \subset \bbR^n$ and $x \in \calX$, a vector $v \in \bbR^n$ is a \emph{tangent vector of $\calX$ at $x$} if there exist sequences $x_k \rightarrow x$ with $x_k \in \calX$ and $\delta_k \rightarrow 0^+$ such that 
$\tfrac{x_k - x}{\delta_k} \rightarrow v$. The set of all tangent vectors at $x$ is called \emph{tangent cone} (or sometimes \emph{Bouligand contingent cone}) and denoted by $T_x \calX$. For every $x$, the set $T_x \calX$ is closed (by definition), non-empty (namely, $0 \in T_x \calX$ always) and a cone in the formal sense (i.e., $v \in T_x \calX \Rightarrow \lambda v \in T_x \calX$ for all $\lambda \geq 0$).

For time-varying domains we require an appropriate generalization of the tangent cone. Time cannot be treated the same as space dimensions since it evolves at a steady rate in one direction. This leads us to the following definition.

\begin{definition} [Temporal tangent vector] \label{def:ttgt_cone}
Given a set-valued map $\calX:\mathbb{R}\rightrightarrows\mathbb{R}^n$ and $x\in \calX(t)$ for some $t$, a vector $v \in \mathbb{R}^n$ is a \emph{temporal tangent vector of $\calX(t)$ at $x$ and time $t$}, if there exist sequences $x_k\rightarrow x$ and $\delta_k \rightarrow 0^+$ such that 
    \begin{equation*}
    x_k \in \calX(t+\delta_k)
    \qquad \text{and} \qquad 
    \frac{x_k-x}{\delta_k}\rightarrow v.
    \end{equation*}
    The set of all temporal tangent vectors at $(x,t)$ is called the \emph{temporal tangent cone} and denoted by $T^t_x\calX$.
\end{definition} 

As for $T_x \calX$, the temporal tangent cone $T^t_x\calX$ is closed, it might however be empty. One of the contributions of this paper is to give a necessary and sufficient condition under which $T^t_x\calX$ is non-empty.

Note that $T^t_x\calX$ is not necessarily a cone in the formal sense, i.e., $v \in T^t_x\calX$ does not imply $\lambda v \in T^t_x\calX$ for every $\lambda \geq 0$. However, if $\calX(t)$ is constant, then Definition~\ref{def:ttgt_cone} reduces to the definition of $T_x \calX$, hence our choice of terminology.

\begin{remark} In viability theory~\cite{aubin_viability_1991, aubin_viability_2011}, $T^t_x\calX$ is usually defined using the one-sided \emph{contingent derivative}, i.e., $v \in T^t_x\calX$ if and only if   
     \begin{equation*}
       \underset{\delta \to 0^+}{\mathrm{lim \ inf}} \ \frac{d(x+\delta v,\calX(t+\delta))}{\delta}=0 
     \end{equation*}
     where $d(\cdot,\calX(t+\delta))$ denotes the point-to-set distance.
 \end{remark}

The following proposition shows that temporal tangent cones are closed under finite unions. This results is important in view of domains that have a piecewise definition.

\begin{prop} \label{prop:union_set}
    Let $\{\calX_i(t)\}_{i=1}^m$ be a finite family of time-varying domains $\calX_i: \bbR \rightrightarrows \bbR^n$. Then, 
    \begin{align*}
        T^t_x \left( \bigcup\nolimits_{i=1}^m \calX_i \right) = \bigcup_{i=1}^m T^t_x \calX_i \, .
    \end{align*}
\end{prop}

\begin{proof} ($\supset$) Immediate, since any sequences $\{x_k\}$ and $\{\delta_k\}$ satisfying Definition~\ref{def:ttgt_cone} (in particular $x_k \in \calX_i(t + \delta_k)$) also satisfy $x_k \in \bigcup_{i=1}^m \calX_i (t + \delta_k)$ and therefore define a temporal tangent vector of $\bigcup_{i=1}^m \calX_i$ at $(x,t)$. 
($\subset$) Let $\{x_k\}$ and $\{\delta_k\}$ define a temporal tangent vector of $\bigcup_{i=1}^m \calX_i$ at $(x,t)$, i.e., $x_k \rightarrow x$, $\delta_k \rightarrow 0^+$ and $x_k \in \bigcup_{i=1}^m \calX_i (t + \delta_k)$ such that $\tfrac{x_k -x}{\delta_k} \rightarrow v$.
Let $\{x_k\}_{\calX_i}$ denote the subsequence of $\{x_k\}$ defined by selecting all elements that lie in $\calX_i(t + \delta_k)$. Since there are only finitely many $\calX_i$ we may choose $i$ such that $\{x'_k\} := \{x_k\}_{\calX_i}$ is an infinite subsequence of $\{x_k\}$. Then, let $\{\delta'_k\}$ the associated subsequence of $\{\delta_k \}$ such that $x'_k \in \calX_i(t + \delta'_k)$. Since any (infinite) subsequence of a converging sequence converges to the same limit value it follows that $\tfrac{x'_k -x}{\delta'_k} \rightarrow v$ and therefore $v \in T^t_x \calX_i$.
\end{proof} 

Next, we show that for basic sets of the form
\begin{align}\label{eq:basic_set}
	\calX(t) := \left\{ x \in \bbR^n \, \middle| \,  \, g(x,t) \leq 0 \right\}
\end{align}
where $g : \bbR^{n} \times \bbR \rightarrow \bbR^m$ is $C^1$ in $x$ and $t$, the temporal tangent cone takes an explicit form. 
For this, we define the \emph{index set of active constraints at $(x,t)$} by $\mathbf{I}(x,t) := \{ i \, | \, g_i (x,t) = 0\}$ and $g_{\mathbf{I}(x,t)}$ as the function obtained from stacking only constraint functions $g_i$ that are active at $(x,t)$.

\begin{prop}\label{prop:constraint_set} Consider the time-varying set of the form~\eqref{eq:basic_set} and assume that $\nabla_x g_{\mathbf{I}(x,t)} (x,t)$ has full rank for every $(x,t)$. Then, the temporal tangent cone is given by
\begin{align} \label{eq:ttc_ex1}
	T^t_x\calX = \left\{ v  \, \middle| \, 
	\nabla g_{\mathbf{I}(x,t)} (x,t) 
	\begin{bmatrix} 
	   v \\ 1 
	\end{bmatrix} \leq 0 \right\} \, .
\end{align}
\end{prop}

\begin{proof}

Without loss of generality, let us show~\eqref{eq:ttc_ex1} for $t=0$. For this, we consider the graph of $\calX|_{t \geq 0}$ defined as
\begin{align*}
	\gph \calX|_{t \geq 0} := \left\{ (x, t) \in \bbR^n \times \bbR \, \middle| \,  \left[ \begin{smallmatrix} \\g(x,t) \\  -t \end{smallmatrix}\right]  \leq 0 \right\}
\end{align*} 
Note that for any $(x,t) \in \gph \calX|_{t \geq 0}$ the Jacobian
\begin{align*}
  \nabla \begin{bmatrix} g_{\mathbf{I}(x,t)}(x,t) \\ -t \end{bmatrix} = \begin{bmatrix} \nabla_x g_{\mathbf{I}(x,t)}(x,t) & \nabla_t g_{\mathbf{I}(x,t)} (x,t) \\ 0 & -1  \end{bmatrix}
  \end{align*}
has full rank since $\nabla g_{\mathbf{I}(x,t)}$ has full rank by assumption. Hence, by the standard result~\cite[Thm 6.31]{rockafellar_variational_1998}, the tangent cone of $\gph \calX|_{t \geq 0}$ for $t=0$ and $x \in \calX(0)$ is
\begin{align*}
	T_{(x,0)} \left( \gph \calX|_{t \geq 0} \right) 
	   := \left\{ v' \, \middle| \,  
	      \nabla 
	    \begin{bmatrix} g_{\mathbf{I}(x,t)}(x,t) \\ -t   \end{bmatrix}  v' \leq 0
	\right\} \, .
\end{align*}

Hence, for every $v' \in T_{(x,0)}$ there exist sequences $\{d_k\}$ and $\{(x_k, \delta_k)\}$ converging to $0$ and $(x,0)$ respectively, such that $(x_k, \delta_k) \in \gph \calX|_{t \geq 0}$ and $\tfrac{(x_k, \delta_k) - (x, 0)}{d_k} \rightarrow v'$. This implies that $x_k \in \calX(\delta_k)$ and $\delta_k \geq 0$. The temporal tangent cone is exactly the subset of $T_{(x,0)} \left( \gph \calX|_{t \geq 0} \right)$ for which $\frac{\delta_k}{d_k} \rightarrow 1$ and therefore the last component of $v'$ is one.
\end{proof}

Proposition~\ref{prop:constraint_set} can be extended to sets incorporating (differentiable) equality constraints $h(x,t ) = 0$ as long as $[ \nabla_x h^T(x,t) \, \nabla_x g_{\mathbf{I}(x,t)}^T (x,t) ]$ has full rank. Furthermore, Propositions~\ref{prop:union_set} and~\ref{prop:constraint_set} can be combined to construct the temporal tangent cone of sets of the form~\eqref{eq:basic_set} where $g$ is only piecewise differentiable in $x$.

In the study of physical systems with saturation it is in general not necessary to explicitly compute the temporal tangent cone since the projection on the temporal tangent cone is a natural phenomenon accomplished by the physics of the system. However, it is necessary for the temporal tangent cone to be well-behaved. This can be accomplished under weaker conditions than $C^1$ differentiability in $t$ as we will show in the next section.

\subsection{Forward Lipschitz continuity}
In the next section we will study projected dynamical systems which are defined by projecting a vector field $f(x)$ onto $T^t_x \calX$ at every $x \in \calX$. For this to be well-defined, we require $T^t_x \calX$ to be non-empty for all $t$ and $x \in \calX(t)$. In order to study when this is the case, we introduce the following definition.

\begin{definition} [Forward Lipschitz continuity] \label{Lipschitz_forward_continuity2}
     A non-empty set-valued map $\calX:\mathbb{R}\rightrightarrows\mathbb{R}^n$ is \emph{forward Lipschitz continuous at $t \in \mathbb{R}$} if there exists a constant $L > 0$ such that for every $\delta \in [0, D)$ for some $D > 0$ one has
    \begin{align}
    \calX(t) \subseteq  \calX(t + \delta)+ \delta L\mathbb{B} \, .
    \label{eq:forw_lip_def}
    \end{align}
    The domain $\calX(t)$ is \emph{forward Lipschitz continuous} if it is forward Lipschitz continuous for all $t \in \bbR$ for the same $L$.
\end{definition} 

In essence, forward Lipschitz continuity precludes the possibility that a time-varying domain shrinks at an unbounded rate. An expansion of the set, on the other hand, can be discontinuous. Fig.~\ref{fig:forw_lip} illustrates the concept for 1-dimensional sets varying over time. 

\begin{figure}[b]
    \centering
    \subfloat[forward Lipschitz]{
        \includegraphics[width=0.42\linewidth]{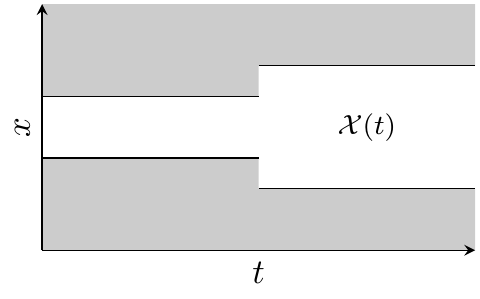}
    } \hspace{.5cm}
    \subfloat[not forward Lipschitz]{
        \includegraphics[width=0.42\linewidth]{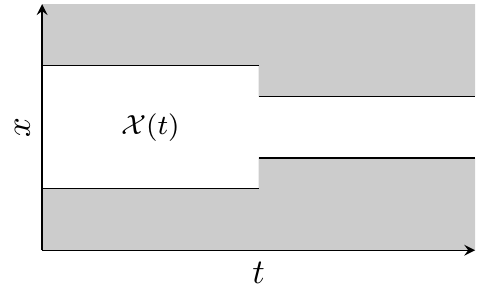}
    }    
    
    \subfloat[forward Lipschitz]{
        \includegraphics[width=0.42\linewidth]{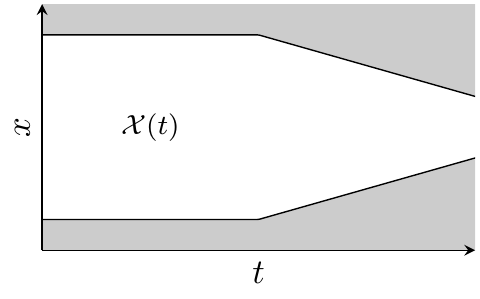}
    } \hspace{.5cm}
    \subfloat[not forward Lipschitz]{
        \includegraphics[width=0.42\linewidth]{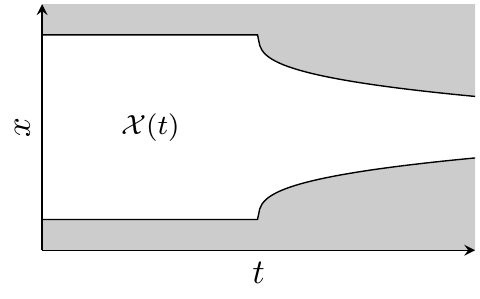}
    }
    \caption{1-dimensional examples of (non-)forward Lipschitz domains}
    \label{fig:forw_lip}
\end{figure}

The following key result shows that the temporal tangent cone is non-empty if the domain is forward Lipschitz continuous.

\begin{thm} \label{thm:non_empty_ttc}
    Consider a non-empty set-valued map $\calX: \bbR \rightrightarrows \mathbb{R}^n$. 
    Given $t$, the temporal tangent cone $T^t_x \calX$ is non-empty for every $x \in \calX(t)$ if $\calX(t)$ is forward Lipschitz continuous at $t$. Furthermore, $T^t_x \calX \cap L \bbB \neq \emptyset$.
\end{thm}

\begin{proof} Assume that $\calX(t)$ is forward Lipschitz continuous at $t$ with constant $L > 0$. Given $x \in \calX(t)$, we construct a temporal tangent vector at $x$ as follows:
Consider any sequence $\delta_k \rightarrow 0^+$ with $\delta_k \in [0, D)$. Since $\calX(t+ \delta)$ is non-empty and~\eqref{eq:forw_lip_def} holds, there exists $x_k \in \calX(t+\delta_k)$ such that $\| x_k-x \| \ \leq \ L \delta_k$ for all $k$.

Hence, the sequence $\tfrac{x_k - x}{\delta_k}$ is bounded. Using Bolzano-Weierstrass, we conclude the existence of a convergent subsequence which satisfies the definition of a temporal tangent vector. Moreover, this temporal tangent vector has norm less or equal to $L$ which proves the second statement.
\end{proof}

The following results shows that forward Lipschitz continuity is preserved under finite unions of sets.

\begin{prop} \label{prop:union_set_2}
    Let $\{\calX_i(t)\}_{i=1}^m$ be a finite sequence of forward Lipschitz continuous domains $\bbR \rightrightarrows \bbR^n$. Then, their union $\bigcup\nolimits_{i=1}^m \calX_i (t)$ is is forward Lipschitz continuous.
\end{prop}

\begin{proof} For every $\calX_i(t)$ let the Lipschitz constants as defined in \eqref{eq:forw_lip_def} be denoted by $L_i$. Then, we have
    \begin{align*}
        \bigcup_{i=1}^m \calX_i (t) \subset \bigcup_{i=1}^m \calX_i (t + \delta )+ \delta \underset{i=1,\ldots, m}{\max} \{L_i\} \bbB \, . \tag*{\qedhere} 
    \end{align*}
\end{proof}

The next result shows that the basic sets of the form~\eqref{eq:basic_set} are forward Lipschitz continuous even if the constraints are only Lipschitz in time. The proof can be found in the appendix.

\begin{prop} \label{prop:constraint_set_2}
Consider the time-varying set
\begin{align*}
	\calX(t) := \left\{ x \subset \bbR^n \, \middle| \, g(x,t) \leq 0 \right\}
\end{align*}
where $g : \bbR^{n} \times \bbR \rightarrow  \bbR^m$ is $C^1$ in $x$ and $\nabla_x g(x,t)$ has full rank for every $(x,t)$, and $g$ is Lipschitz continuous in $t$. Then, $\calX(t)$ is forward Lipschitz continuous.
\end{prop}

Proposition~\ref{prop:constraint_set_2} can be easily generalized to the case where only $g_{\mathbf{I}(x)}(x)$ requires full rank instead of $g(x)$.
Similarly to Proposition~\ref{prop:constraint_set},  Proposition~\ref{prop:constraint_set_2} can also be extended to sets incorporating equality constraints.
Furthermore, Propositions~\ref{prop:union_set_2} and~\ref{prop:constraint_set_2} can be combined to show forward Lipschitz continuity of piecewise defined sets.

\begin{example} \label{ex:non_forw_lip}
 	Consider the time-varying domains given by
 	\begin{align*}
 		\calX_a (t) &= \{x \in \bbR^2 \, | \, x_2 \geq 0, \, x_2 \leq | x_1 | - t \} \\
 	 	\calX_b (t) &= \{x \in \bbR^2 \, | \, x_2 \geq 0, \, x_2 \leq x_1^2 - t \} \, .
 	\end{align*}
 	
 	As illustrated in Fig.~\ref{fig:non_forw_lip1}, the set $\calX_a(t)$ is delimited by a fixed lower bound on $x_2$ and a ``wedge'' that moves vertically down. At $t=0$ the wedge touches the lower limit and for all $t > 0$ the set $\calX_a$ is separated into two parts. The behavior of $X_b(t)$ is the same except that the wedge is replaced by a parabola (Fig.~\ref{fig:non_forw_lip2}).
 	In both cases we are interested in the temporal tangent cone at the origin $x = 0$ at time $t=0$.
 	
 	Informally, a particle residing at $(x,t)=(0,0)$ can remain in $\calX_a$ only by moving at a minimum horizontal velocity that is large enough to evade the moving wedge. This set of admissible velocities, i.e., the temporal tangent vectors, is given by the red hatched area in the second panel of Fig.~\ref{fig:non_forw_lip1}.

  	Formally, the set $\calX_a(t)$ is the union of two forward Lipschitz continuous sets $\{x \, | \, x_1 \leq 0, 0 \leq x_2 \leq - x_1 + t \}$ and $\{x \, | \, x_1 \geq 0, 0 \leq x_2 \leq x_1 + t \}$, both of which satisfy the requirements of Proposition~\ref{prop:constraint_set}.
 	Using Proposition~\ref{prop:union_set}, the temporal tangent cone at $(x,t) = (0,0)$ is given by
 	\begin{align*}
 		T^0_0 \calX_a = \{ v \in \bbR^2 \, | \,  v_2 \leq \| v_1 \| - 1, \, v_2 \geq 0 \} \,.
 	\end{align*}
    	 
    In the case of $\calX_b(t)$, it is not possible to leave the $(x,t)=(0,0)$ at finite velocity while guaranteeing feasibility. Hence, $T^0_0 \calX_b $ is empty. In fact, $\calX_b(t)$ is not forward Lipschitz continuous at $t = 0$. To see this, consider the point $p(t) := (\sqrt{t}, 0)$ for $t \geq 0$ that lies at the intersection between the two constraints. In particular, $\tfrac{d}{dt}p(t) \rightarrow \infty$ as $t \rightarrow 0^+$.
    \begin{figure}[tb]
        \centering
        \subfloat[$\calX_a(t)$ is forward Lipschitz for all $(x,t)$. The temporal tangent cone at $(x,t)=(0,0)$ is indicated in red.]{
        \label{fig:non_forw_lip1}
            \includegraphics[width=0.45\linewidth]{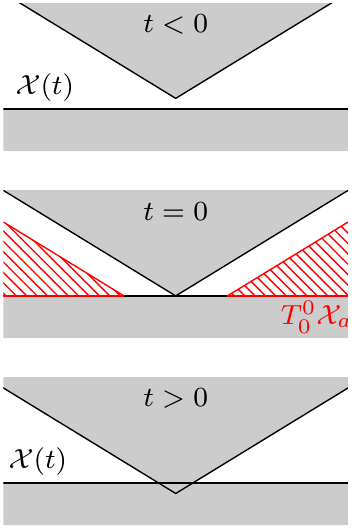}
        } \hspace{.02\linewidth}
        \subfloat[The set $\calX_b(t)$ is not forward Lipschitz at $(x,t)=(0,0)$. $T_0^0\calX_b$ is empty.]{
        	\label{fig:non_forw_lip2}
            \includegraphics[width=0.45\linewidth]{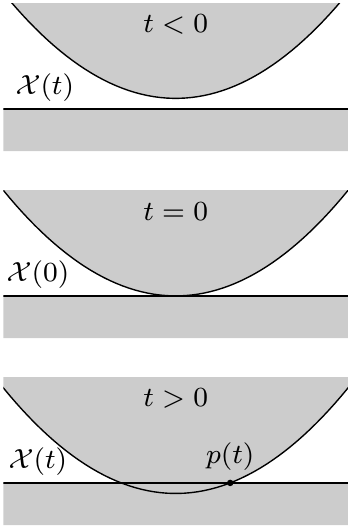}
        }
        \caption{Sets $\calX_a(t)$ and $\calX_b(t)$ (white areas) for Example~\ref{ex:non_forw_lip}}
        \label{fig:non_forw_lip}
    \end{figure}
\end{example}

Example~\ref{ex:non_forw_lip} highlights that the condition on the rank of $\nabla_x g_{\mathbf{I}(x,t)}(x,t)$ in Proposition~\ref{prop:constraint_set_2} is generally necessary for forward Lipschitz continuity.

\section{Projected Dynamical Systems}\label{sec:pds}
Next, we define the notion of a projected dynamical system for a given vector field $f: \bbR^n \times \bbR \rightarrow \bbR^n$ and a time-varying domain $\calX(t)$. Instead of a vector field, it is possible to consider a set-valued map $F: \bbR^n \times \bbR \rightrightarrows \bbR^n$ under suitable assumptions~\cite{cornet_existence_1983}.

\begin{definition}
Given a set-valued map $\calX: \bbR \rightrightarrows \bbR^n$ and a vector field $f: \bbR^n \times \bbR \rightarrow \bbR^n$, we define the projected set-valued map
\begin{align*}
\Pi_\calX f:  \bbR^n \times \bbR & \rightrightarrows \bbR^n \\ 
 (x,t) &\mapsto \underset{v \in T^t_x \calX}{\arg \min} \| v - f(x,t) \| \, .
\end{align*}
\end{definition}

Note that $\Pi_\calX f(x,t)$ is closed and non-empty as long as $T^t_x\calX$ is non-empty. Hence, we consider \emph{projected dynamical systems} that are defined by the initial value problem
\begin{equation} \label{eq:pds_ivp}
    \dot x \in \Pi_\calX f(x,t) \,, \qquad x(0) = x_0 \in \calX(0)\, .
\end{equation}

Traditionally, a \emph{Carath\'eodory solution} of \eqref{eq:pds_ivp} is defined as an absolutely continuous function $x: [0, D) \rightarrow \calX$ for some $D>0$ and $x(0) = x_0$ that satisfies $\dot x \in \Pi_\calX f(x,t)$ almost everywhere (i.e., for almost all $t \in [0,D)$). 

The existence of Carath\'eodory solutions is guaranteed under additional assumptions on $\calX$ and $f$. Namely, in general $f$ needs to be measurable in $t$ and locally bounded in $x$ (or Lipschitz for global existence,i.e., $D \rightarrow \infty$).
In the stationary case, existence results are known if $\calX$ is convex \cite{nagurney_projected_1996} or Clarke regular~\cite{cornet_existence_1983}. An example for which no Carath\'eodory solution exists can be found in~\cite{hauswirth_projected_2018}.

In this paper, we study the broader class of Krasovskii solutions which we show to exist for general forward Lipschitz continuous $\calX(t)$ that have a closed graph.

\begin{definition}[Krasovskii regularization]
    Given a closed, locally bounded set-valued map $F: \mathbb{R}^n \times \mathbb{R} \rightrightarrows \mathbb{R}^n$, its Krasovskii regularization is defined as the set-valued map given by
    \begin{align*}
        K[F]:\bbR^n \times\bbR & \rightrightarrows\mathbb{R}^n  
        \\ (x, t) & \mapsto   \bigcap_{\epsilon > 0} \cocl F(x + \epsilon \mathbb{B}, t )
    \end{align*} 
\end{definition}

Informally, the Krasovskii regularization of $F$ assigns to $x$ the closed convex hull of all limit values of $F$ at $x$.

Given a set valued map $F: \mathbb{R}^n \times \bbR \rightrightarrows\mathbb{R}^n$, an absolutely continuous function $x:[0,D)\rightarrow\mathbb{R}^n$ for some $D>0$ and $x(0)=x_0 \in \calX(0)$ is a \emph{Krasovskii solution} to the initial value problem $\dot{x}\in F(x,t)$ with $x(0)=x_0$ if it satisfies $\dot{x}\in K[F](x,t)$ almost everywhere, i.e., $x(t)$ is a Carath\'eodory solution to the regularized inclusion.
Hence, our main result on the existence of Krasovskii solutions reads as follows.

\begin{thm}[Existence of Krasovskii solutions] \label{thm:main_exist}
Consider
\begin{itemize}
 \item a non-empty, forward Lipschitz continuous domain $\calX: \bbR \rightrightarrows \bbR^n$ with a closed graph, and
 \item a vector field $f: \bbR^n \times \bbR  \rightarrow \bbR^n$ that is Lipschitz continuous in $x$ and measurable in $t$.
\end{itemize}
Then, for any $x_0 \in \calX(0)$ there exists a Krasovskii solution $x: [0, \infty) \rightarrow \bbR^n$ to the problem
\begin{align*}
     \dot x \in \Pi_\calX f(x, t) \, , \quad x(0) = x_0
\end{align*}
satisfying $x(t) \in \calX(t)$ for all $t \in [0, \infty)$.
\end{thm}

\subsection{Proof of Theorem~\ref{thm:main_exist}}

We prove Theorem~\ref{thm:main_exist} by showing that 
the Krasovskii regularized problem satisfies the conditions of a more fundamental existence theorem from~\cite{tallos_viability_1991} (see~\cite{haddad_monotone_1981} for a similar result). The challenge consists in the fact the regularization $K[ \Pi_\calX f ](x,t)$ is affected by the properties of $\calX(t)$, e.g., its forward Lipschsitz continuity.

In the following, a set-valued map $F: \bbR^n \times \bbR \rightrightarrows \bbR^n$ is \emph{integrably bounded} if there exists a locally integrable function $\ell(t)$ such that for almost all $t \in \bbR$ and for every $x \in \bbR^n$ it holds that $F(t,x) \subset \ell(t)(1 + \|x \|) \mathbb{B}$.

\begin{thm}{\cite[Thm 3]{tallos_viability_1991}} \label{thm:base}
    Consider $\calX : [0, \infty) \rightrightarrows \bbR^n$ with closed graph and $F: \bbR^n \times \bbR \rightarrow \bbR^n$ integrably bounded, (Lebesgue) measurable in $t$, with closed graph in $x$, and with non-empty, convex compact values. If for almost every $t \in [0, \infty)$ and every $x \in \calX(t)$ one has
    \begin{align}
        F(x,t) \cap T^t_x \calX \neq \emptyset\, , \label{eq:viab_cond}
    \end{align}
then for every $x_0 \in \calX(0)$ there exists an absolutely continuous function $x(t): [0, \infty) \rightarrow \bbR^n$ such that $x(0) = x_0$, $x(t) \in \calX$ for all $t \in [0,\infty)$, and $\dot x(t) \in F(x(t), t)$ for almost all $t \in [0, D)$.
\end{thm}

In our case $F(x,t) = K \left[ \Pi_\calX f \right] (x, t)$ is closed and convex by definition. Furthermore, $\Pi_\calX f(x,t) \subset K\left[\Pi_\calX f\right](x,t)$ holds by definition of $K\left[\Pi_\calX f\right]$. Namely, using Theorem~\ref{thm:non_empty_ttc}, $K\left[\Pi_\calX f\right](x,t)$ is non-empty since $\calX(t)$ is forward Lipschitz continuous and therefore $T^t_x \calX$ is non-empty. Moreover, \eqref{eq:viab_cond} is satisfied for any vector field $f$, again, by definition of $\Pi_\calX f$.

Next, we show that  $K[\Pi_\calX f](x,t)$ has a closed graph in $x$ given $t$. For this, we first take the closure of the graph of $\Pi_\calX f(x,t)$ in $x$ which defines a new set-valued map $\overline{\Pi_\calX f}(x,t)$, i.e.,
\begin{align*}
	\gph_x \overline{\Pi_\calX f}(\cdot, t) = \cl \{ (x, v) \, | \, v \in \Pi_\calX f(x,t) \}\, .
\end{align*}

Given $t$, $\overline{\Pi_\calX f}(x,t)$ is non-empty and compact for every $x$, and hence~\cite[Lem 16, \S 6]{filippov_differential_1988} implies that $\co \overline{\Pi_\calX f}(x,t) = K[\Pi_\calX f](x,t)$ has a closed graph for every $t$.

It remains to show that $K[\Pi_\calX f] (x,t)$ is integrably bounded and measurable in $t$. This follows from standard results as we show in the following.

In analogy to single-valued maps, a set-valued map $F: \bbR \rightrightarrows \bbR^n$ is \emph{(Lebesgue) measurable} if for every open set $U \subset \bbR^n$, the preimage $F^{-1}(U) := \{t \in \bbR \, | \, F(t) \cap U \neq \emptyset \}$ is (Lebesgue) measurable. 
Since $\calX$ has a closed graph, it follows that it is measurable~\cite[Ex 14.9]{rockafellar_variational_1998}.
 
If both $f$ and $\calX$ are measurable in $t$, it follows from~\cite[Thm 14.26]{rockafellar_variational_1998} and~\cite[Ex 14.17]{rockafellar_variational_1998} that the projected map $\Pi_\calX f(x,t)$ is measurable in $t$. Finally, since taking the closure ~\cite[Prop 14.2]{rockafellar_variational_1998} and taking the convex hull~\cite[Ex 14.2]{rockafellar_variational_1998} of a set-valued map preserve measurability, it follows that $K[\Pi_\calX f(x,t)]$ is measurable in $t$.

Integrable boundedness follows from the fact that if $f(x, t)$ is Lipschitz continuous in $x$ and measurable in $t$ then $\Pi_\calX f(x,t)$ as well as $K[\Pi_\calX f](x,t)$ are integrably bounded by non-expansiveness of the projection operator and taking the convex hull, respectively.
Namely, $\|\Pi_\calX f(x,t) - \Pi_\calX 0(x,t) \| \leq \| f(x,t) - 0 \|$ where $\Pi_\calX 0(x,t)$ denotes the projection of the zero vector on $T_x^t \calX$. Since, by Theorem~\ref{thm:non_empty_ttc}, there exists $v \in T_x^t \calX$ with $\| v \| \leq L$ where $L > 0$ is the global Lipschitz constant of $\calX(t)$, it follows that $\| \Pi_\calX 0(x,t) \| \leq L$. Therefore, $\| \Pi_\calX f(x,t) \| \leq \| \Pi_\calX 0(x,t) \| + \|\Pi_\calX f(x,t) - \Pi_\calX 0(x,t) \|  \leq L + \| f(x,t)\|$ and consequently $\| K[\Pi_\calX f] (x,t) \| \leq L + \| f(x,t)\|$.

Hence, Theorem~\ref{thm:base} can be applied which concludes the proof of Theorem~\ref{thm:main_exist}.

\section{Application to Power Systems} \label{sec:ps_ex}
As a motivation for the results presented in this work, we consider the application of real-time operation of power systems. 
We consider a steady-state model of a power transmission grid, where the dynamics of the transmission lines, of the generators, and of the low-level frequency and voltage controllers, are assumed to be at steady state which can be described algebraically. 
As represented in Fig.~\ref{fig:realtimeopt}, 
the problem of real-time optimization (for the details of which we refer to~\cite{hauswirth_online_2017,tang_real-time_2017,dallanese_optimal_2018,bernstein_realtime_2018}) consists in updating, based on feedback measurements of the grid state, the set-point of the generators of the grid in order to maximize some opportune utility (e.g., economical generation).
A number of local low-level controllers exhibit input saturation, and therefore cannot track set-points which do not belong to some range. Examples for this include droop controllers that stabilize the system frequency by injecting more or less mechanical power unless the limits on power generation are reached, and automatic voltage regulators that control the voltage at a given bus by injecting or absorbing reactive power unless the limit on reactive power generation is reached.
The presence of these constraints induces a feasible region that is defined by different modes depending on whether individual controllers are saturated.
The update of set-points induces a trajectory on this ``partitioned'' feasible region which is effectively modelled using projected dynamical systems.
As some constraints (e.g., maximal power generation, load power demand) change in time, the results proposed in this paper become useful to guarantee the well-posedness of this abstraction, and the existence of a trajectory of the closed-loop system.

\begin{figure}[tb]
    \centering
    \includegraphics[width=\columnwidth]{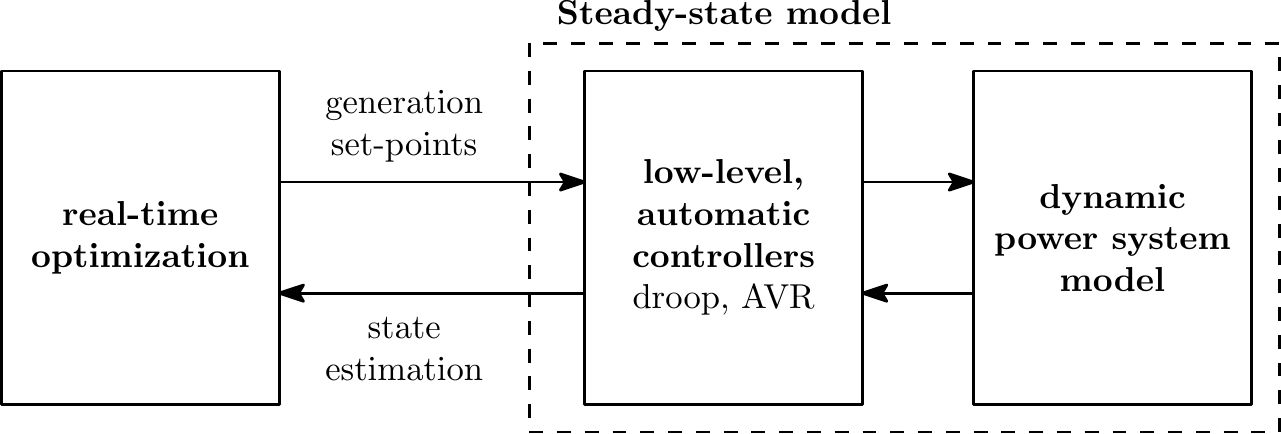}
    \caption{Real-time optimization in power systems}
    \label{fig:realtimeopt}
\end{figure}

For the sake of illustration, we consider a simplified case (Fig.~\ref{fig:twobusgrid}) in which a voltage-regulated generator and a time-varying active load $p^\text{L}$ are connected through a transmission line to an infinite bus (representing the rest of the grid).

Hence, the AC power flow equations governing the physical flow of energy can be reduced to
\begin{align}\label{eq:ex3_pf}
h(x, t) &= \begin{bmatrix}
		p^\text{G} - p^\text{L}(t)  - v \sin(\theta_2) \\
		q^\text{G}  + v \cos(\theta_2) - v^2 \\
\end{bmatrix} = 0  
\end{align}
where $x = \begin{bmatrix}
	p^\text{G} & q^\text{G} &  v & \theta
\end{bmatrix}$. 

The local voltage controller of the generator regulates the voltage $v$ to $1$ under normal operating conditions, and consequently varies its reactive power injection $q^\text{G}$. 
However, this reactive power generation is limited by $ \underline{q} \leq q^\text{G} \leq \overline{q}.$
If either limit is active, the voltage controller saturates and the voltage deviates from the set-point.
In steady-state modeling terms, the bus behaves as a \emph{PV bus} when the generated reactive power is within limits, and as a \emph{PQ bus} when in saturation.

The resulting feasible domain is illustrated in Fig.~\ref{fig:twobusgrid}, and is the union of three different regimes given by
\begin{align*}
	\calX_1(t) &: = \{ (x,t) \, | \, h(x,t) = 0, \, v = 1, \underline{q} \leq q^\text{G} \leq \overline{q} \} \\	
	\calX_2(t) &: = \{ (x,t) \, | \, h(x,t) = 0, \, v \ge 1, \underline{q} = q^\text{G}\} 
	\\	
	\calX_3(t) &: = \{ (x,t) \, | \, h(x,t) = 0, \, v \le 1, \overline{q} = q^\text{G}\}  \, .
\end{align*}

\begin{figure}[tb]
    \centering
    \raisebox{3mm}{\includegraphics[width=0.39\columnwidth]{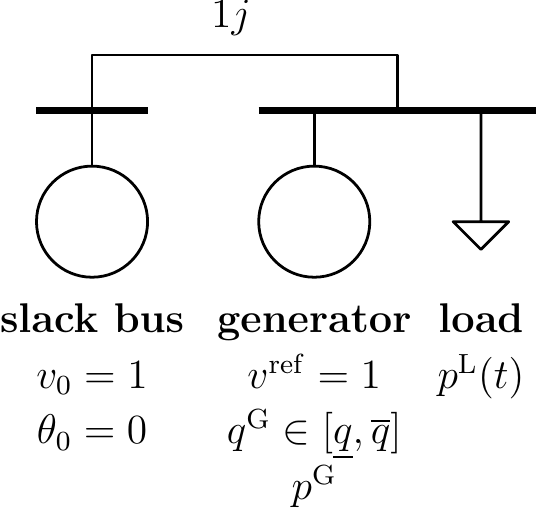}}
    \includegraphics[width=0.59\columnwidth]{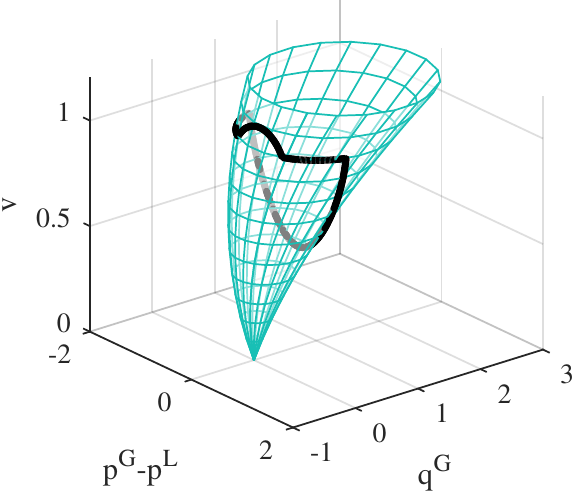}
    \caption{Two-bus example. The thick line represents the feasible region induced by the low-level voltage regulation mechanism of the generator.}
    \label{fig:twobusgrid}
\end{figure}

Given these modeling assumptions, the feedback optimization scheme described in Fig.~\ref{fig:realtimeopt} induces a closed-loop system that needs to evolve on $\calX(t) := \calX_1(t) \cup \calX_2(t) \cup \calX_3(t)$. In order to guarantee the existence of its trajectories Theorem~\ref{thm:main_exist} requires forward Lipschitz continuity of $\calX(t)$ which can be established using Propositions~\ref{prop:union_set_2} and~\ref{prop:constraint_set_2}.

The primary obstacle to concluding that Krasovskii solutions exist for any measurable and Lipschitz continuous vector field induced on $\calX(t)$ is the condition that the active constraint function together with the equality constraint
$h(x,t)=0$ need to have full rank. This needs to be verified for each set $\calX_i(t)$ separately.

In the present example, we find by inspection that such a rank condition holds everywhere except at the bifurcation point that defines voltage collapse~\cite{Sauer_Jacobian_1990}, a correspondence that we expect to be valid for general networks.
However, in a general network, the rank constraint of Proposition~\ref{prop:constraint_set_2} can also fail to hold for other points of the domain where the sensitivity of the grid state with respect to a saturated variable becomes infinite. In practical terms, the non-existence of a feasible trajectory can be interpreted as a lack of control authority:  no finite control effort will suffice to maintain the grid state inside the desired bounds, or to track the prescribed reference. A full characterization of these regimes remains an open question.

\section{Conclusion}
We derived conditions for the existence of solutions of the autonomous dynamics that emerge from the feedback optimization of a static plant with time-varying constraints.
These conditions, based on the concept of forward Lipschitz continuity of the feasible region, have then be translated into conditions on the optimization problem that is being solved.

In the context of real-time optimization of power systems, this analysis allows to identify critical operating regimes in which feasible closed-loop system trajectories are not guaranteed to exist under steady-state modeling assumptions.
In doing so, we recover voltage instability boundaries (not surprisingly) but we can also identify configurations where we lose controllability of the power system state.
Such well-posedness conditions should be considered in the design of real-time feedback optimization laws with the goal of maintaining the power system sufficiently far from these critical configurations. 

\bibliographystyle{IEEEtran}
\bibliography{IEEEabrv,bibliography_static}

\begin{thebibliography}{10}
\providecommand{\url}[1]{#1}
\csname url@samestyle\endcsname
\providecommand{\newblock}{\relax}
\providecommand{\bibinfo}[2]{#2}
\providecommand{\BIBentrySTDinterwordspacing}{\spaceskip=0pt\relax}
\providecommand{\BIBentryALTinterwordstretchfactor}{4}
\providecommand{\BIBentryALTinterwordspacing}{\spaceskip=\fontdimen2\font plus
\BIBentryALTinterwordstretchfactor\fontdimen3\font minus
  \fontdimen4\font\relax}
\providecommand{\BIBforeignlanguage}[2]{{%
\expandafter\ifx\csname l@#1\endcsname\relax
\typeout{** WARNING: IEEEtran.bst: No hyphenation pattern has been}%
\typeout{** loaded for the language `#1'. Using the pattern for}%
\typeout{** the default language instead.}%
\else
\language=\csname l@#1\endcsname
\fi
#2}}
\providecommand{\BIBdecl}{\relax}
\BIBdecl

\bibitem{nelson_integral_2017}
Z.~E. Nelson and E.~Mallada, ``An integral quadratic constraint framework for
  real-time steady-state optimization of linear time-invariant systems,''
  \emph{arXiv:1710.10204 [cs, math]}, Oct. 2017.

\bibitem{bolognani_distributed_2015}
S.~Bolognani, G.~Cavraro, R.~Carli, and S.~Zampieri, ``Distributed reactive
  power feedback control for voltage regulation and loss minimization,''
  \emph{{IEEE} Trans. Autom. Control}, vol.~60, no.~4, pp. 966--981, Apr. 2015.

\bibitem{bolognani_distributed_2013-1}
S.~Bolognani and S.~Zampieri, ``A distributed control strategy for reactive
  power compensation in smart microgrids,'' \emph{{IEEE} Trans. Autom.
  Control}, vol.~58, no.~11, pp. 2818--2833, 2013.

\bibitem{gan_online_2016}
L.~Gan and S.~Low, ``An {{Online Gradient Algorithm}} for {{Optimal Power
  Flow}} on {{Radial Networks}},'' \emph{{IEEE} J. Sel. Areas Commun.},
  vol.~34, no.~3, pp. 625--638, Mar. 2016.

\bibitem{dallanese_photovoltaic_2016}
E.~Dall'Anese, S.~V. Dhople, and G.~B. Giannakis, ``Photovoltaic {{Inverter
  Controllers Seeking AC Optimal Power Flow Solutions}},'' \emph{{IEEE} Trans.
  Power Syst.}, vol.~31, no.~4, pp. 2809--2823, Jul. 2016.

\bibitem{dallanese_optimal_2018}
E.~Dall'Anese and A.~Simonetto, ``Optimal power flow pursuit,'' \emph{{IEEE}
  Trans. Smart Grid}, vol.~9, no.~2, pp. 942--952, Mar. 2018.

\bibitem{hauswirth_online_2017}
A.~Hauswirth, A.~Zanardi, S.~Bolognani, F.~D{\"o}rfler, and G.~Hug, ``Online
  {{Optimization}} in {{Closed Loop}} on the {{Power Flow Manifold}},'' in
  \emph{{{IEEE PES PowerTech}}}, 2017.

\bibitem{tang_real-time_2017}
Y.~Tang, K.~Dvijotham, and S.~Low, ``Real-time {{Optimal Power Flow}},''
  \emph{{IEEE} Trans. Smart Grid}, vol.~8, no.~6, pp. 2963--2973, Nov 2017.

\bibitem{bernstein_realtime_2018}
A.~Bernstein and E.~{Dall'Anese}, ``Real-time feedback-based optimization of
  distribution grids: A unified approach,'' \emph{arXiv:1711.01627 [math]},
  Jan. 2018.

\bibitem{hauswirth_projected_2016}
A.~Hauswirth, S.~Bolognani, G.~Hug, and F.~D{\"o}rfler, ``Projected gradient
  descent on riemannian manifolds with applications to online power system
  optimization,'' in \emph{54th Annual Allerton Conf. Communication, Control,
  and Computing}, 2016.

\bibitem{cherukuri_convergence_2015}
A.~Cherukuri, E.~Mallada, and J.~Cort{\'e}s, ``Convergence of {{Caratheodory}}
  solutions for primal-dual dynamics in constrained concave optimization,'' in
  \emph{Proc. Conf. Control and its Applications}, 2015.

\bibitem{nagurney_projected_1996}
A.~Nagurney and D.~Zhang, \emph{Projected {{Dynamical Systems}} and
  {{Variational Inequalities}} with {{Applications}}}, 1st~ed., ser.
  International Series in Operations Research \& Management Science.\hskip 1em
  plus 0.5em minus 0.4em\relax {Springer Science \& Business Media}, 1996,
  no.~2.

\bibitem{cornet_existence_1983}
B.~Cornet, ``Existence of slow solutions for a class of differential
  inclusions,'' \emph{J. Mathematical Analysis and Applications}, vol.~96,
  no.~1, pp. 130--147, Oct. 1983.

\bibitem{henry_existence_1973}
C.~Henry, ``An existence theorem for a class of differential equations with
  multivalued right-hand side,'' \emph{J. Mathematical Analysis and
  Applications}, vol.~41, no.~1, pp. 179--186, Jan. 1973.

\bibitem{lygeros_dynamical_2003}
J.~Lygeros, K.~H. Johansson, S.~N. Simic, J.~Zhang, and S.~S. Sastry,
  ``Dynamical {{Properties}} of {{Hybrid Automata}},'' \emph{{IEEE} Trans.
  Autom. Control}, vol.~48, no.~1, pp. 2--17, Jan. 2003.

\bibitem{liberzon_switching_2003}
D.~Liberzon, \emph{Switching in {{Systems}} and {{Control}}}, ser. Systems \&
  Control: Foundations \& Applications.\hskip 1em plus 0.5em minus 0.4em\relax
  Basel: {Birkh{\"a}user}, 2003.

\bibitem{rockafellar_variational_1998}
R.~T. Rockafellar and R.~J.-B. Wets, \emph{Variational Analysis}, 3rd~ed.\hskip
  1em plus 0.5em minus 0.4em\relax Berlin Heidelberg, Germany: {Springer},
  1998.

\bibitem{aubin_viability_1991}
J.~P. Aubin, \emph{Viability Theory}, ser. Systems \& Control: Foundations \&
  Applications.\hskip 1em plus 0.5em minus 0.4em\relax Boston, MA: {Springer},
  1991.

\bibitem{aubin_viability_2011}
J.-P. Aubin, A.~M. Bayen, and P.~Saint-Pierre, \emph{Viability {{Theory}}:
  {{New Directions}}}, 2nd~ed.\hskip 1em plus 0.5em minus 0.4em\relax Berlin
  Heidelberg: {Springer}, 2011.

\bibitem{hauswirth_projected_2018}
A.~{Hauswirth}, S.~{Bolognani}, and F.~{D{\"o}rfler}, ``{Projected Dynamical
  Systems on Irregular, Non-Euclidean Domains for Nonlinear Optimization},''
  \emph{arXiv:1809.04831 [cs, math]}, Sep. 2018.

\bibitem{tallos_viability_1991}
P.~Tallos, ``Viability {{Problems}} for {{Nonautonomous Differential
  Inclusions}},'' \emph{SIAM J. Control and Optimization}, vol.~29, no.~2, pp.
  253--263, Mar. 1991.

\bibitem{haddad_monotone_1981}
G.~Haddad, ``Monotone trajectories of differential inclusions and functional
  differential inclusions with memory,'' \emph{Israel J. Math.}, vol.~39, no.
  1-2, pp. 83--100, Mar. 1981.

\bibitem{filippov_differential_1988}
A.~F. Filippov, \emph{Differential {{Equations}} with {{Discontinuous Righthand
  Sides}}}, ser. Mathematics and its Applications (Soviet Series).\hskip 1em
  plus 0.5em minus 0.4em\relax Dordrecht, The Nederlands: {Springer}, 1988.

\bibitem{Sauer_Jacobian_1990}
P.~W. Sauer and M.~A. Pai, ``Power system steady-state stability and the
  load-flow jacobian,'' \emph{{IEEE} Trans. Power Syst.}, vol.~5, no.~4, pp.
  1374--1383, Nov. 1990.

\end{thebibliography}

\appendix

\subsection{Proof of Proposition~\ref{prop:constraint_set_2}}

For the proof of Proposition~\ref{prop:constraint_set_2} we require two preceding lemmas. First, we essentially show that a map with full rank has a lower bounded derivative.

\begin{lemma}\label{lem:diff_dist}
    Let $g : \bbR^{n} \rightarrow  \bbR^m$ be $C^1$ and $\nabla g(x)$ have full rank for all $x$. 
    Then, for every $v \in \bbR^n$ with $\| v \|=1$ and $v \notin \ker \nabla g(x)$ there exists $A > 0$ and $L > 0$ such that
     \begin{align} \label{eq:bound1}
      \alpha L \leq \| g (x + \alpha v) - g (x) \|
     \end{align}
     for all $0 \leq \alpha \leq A$
\end{lemma}
\begin{proof}
A Taylor expansion of $g$ at $x$ yields
\begin{align*}
    \| g(x + \alpha v) -  g(x) \|  & =  \left \| \alpha \nabla g(x)(v) + \mathcal{O}(\alpha^2) \right\| \\
    & = \alpha \left \| \nabla g(x)(v) + \mathcal{O}(\alpha) \right\|
\end{align*}
 where $\mathcal{O}(\cdot)$ denotes higher order terms. Since by assumption $\nabla g(x)(v) \neq 0$ it follows that for $\alpha$ small enough there exists $L > 0$  such that $L \leq \| \nabla g(x)(v) + \mathcal{O}(\alpha) \|$ and therefore~\eqref{eq:bound1} holds.
\end{proof}

The next lemma shows that under the full rank assumption on $g$, the norm of the constraint violation (as measured by the the value of $g$) is lower bounded by the distance from the feasible set.

\begin{lemma} \label{lem:diff_dist2}
 Let $\calX := \{ x \, | \, g(x) \leq 0 \}$ where $g : \bbR^{n} \rightarrow  \bbR^m$ is $C^1$ and $\nabla g(x)$ has full rank for all $x \in \calX$. Then, there exists a neighborhood $\mathcal{Y}$ of $\calX$ such that for any $y \in \mathcal{Y}$ and $x \in {\arg\min} \{ \| y - \widetilde{x} \|\, | \, \widetilde{x} \in \calX\}$, there exists $L > 0$ such that for every $y \in \mathcal{Y}$ we have
      \begin{align*}
     L \| y - x \| \leq  \| g_{\mathbf{I}(x)}(y) \| \, .
     \end{align*}
 \end{lemma} 

\begin{proof}
    For any $y$ and a projection $x \in {\arg\min} \{ \| y - \widetilde{x} \|\, | \, \widetilde{x} \in \calX\}$, the vector $y-x$ lies in the normal cone $
        N_x \calX := \{v \, | \, v = \sum\nolimits_{i \in \mathbf{I}(x)} \alpha_i \nabla g_i^T(x), \, \alpha_i \geq 0 \} $ and therefore the span of $\nabla g_{\mathbf{I}(x)}^T(x)$.
    As a consequence of the fundamental theorem of linear algebra this implies that $y-x \notin \ker \nabla g_{\mathbf{I}(x)}(x)$ and Lemma~\ref{lem:diff_dist} is applicable to the function $g_\mathbf{I(x)}$ for $y$ in a neighborhood $\mathcal{Y}$ of $\calX$.
\end{proof}

Thus the proof of Proposition~\ref{prop:constraint_set_2} concludes by showing that a point $y$ that is feasible at time $t$ will have bounded constraint violation at at time $t+\delta$. Using Lemma~\ref{lem:diff_dist2} this bounded constraint violation translates into a bounded distance to the feasible set at $t+\delta$.

\begin{proof}[Proof of Proposition~\ref{prop:constraint_set_2}]
    Let $y \in \calX(t)$. By Lipschitz continuity in $t$ we have for $\delta$ small enough
    \begin{align}\label{eq:proof_lip}
    	 \| g_i (y, t + \delta) - g_i (y, t) \| = \| g_i (y, t + \delta)\| \leq \ell \delta 
    \end{align}
    for every $i \in \mathbf{I}(y,t)$ and for some $\ell > 0$ since $g_i (y, t) = 0$.
    
    Next, assume that $\delta$ is small enough such that $y$ lies in a neighborhood $\mathcal{Y}$ of $\calX(t + \delta)$ for which Lemma~\ref{lem:diff_dist2} is applicable, i.e., that there exists $L > 0$ such that for every $y \in \mathcal{Y}$ and $x \in {\arg\min} \{ \| y - \widetilde{x} \|\, | \, \widetilde{x} \in \calX\}$ it holds that 
    \begin{align}\label{eq:proof_lem}
    	L \| y - x \| \leq \| g_{\mathbf{I}(x, t + \delta)}(y,t+\delta)  \|\, .
    \end{align}
    
    Finally, we need to show that for $\delta$ small enough  the active constraints at $(x,t+\delta)$ are also active at $(y,t)$. This is a consequence of the continuity of $g$ (in both $x$ and $t$). Namely, at $(y,t)$ all inactive constraints are non-zero, i.e., $g_i(y,t) \neq 0$ for all $i \notin \mathbf{I}(y,t)$. Hence, they are also non-zero in a neighborhood (in time and space) of $(y,t)$. By choosing $\delta$ small enough, $(x,t+\delta)$ is in that neighborhood and therefore all constraints that are inactive at $(y, t)$ are also inactive at $(x,t+\delta)$ which implies that $\mathbf{I}(x, t+\delta) \subset \mathbf{I}(y, t)$.

Hence, we can combine~\eqref{eq:proof_lip} and~\eqref{eq:proof_lem} such that for some $L' > 0$ small enough $\delta$ we have for any $y \in \calX(t)$ and $x \in \calX(t+ \delta)$ a projection of $y$ onto $\calX(t+\delta)$ that 
    \begin{align*}
    	\| y - x \|  \leq L' \delta 
    \end{align*}
    which implies that~\eqref{eq:forw_lip_def} holds therefore $\calX(t)$ is Lipschitz forward continuous.
\end{proof}
\end{document}